\documentclass{amsart}

\usepackage{amsfonts,amssymb,amsmath,amsthm,amscd,enumerate}
\usepackage{latexsym}
\usepackage{euscript}
\usepackage{color}

\newtheorem{theorem}{Theorem}[section]

\newtheorem{lemma}[theorem]{Lemma}
\newtheorem{corollary}[theorem]{Corollary}
\newtheorem{proposition}[theorem]{Proposition}

\title{Products of several commutators in a Lie nilpotent associative algebra}

\author{Galina Deryabina}

\address{Department of Computational Mathematics and Mathematical Physics (FS-11), Bauman Moscow State Technical University, 2-nd Baumanskaya Street, 5, 105005 Moscow, Russia}

\email{galina\_deryabina@mail.ru}

\author{Alexei Krasilnikov}

\address{Departamento de Matem\'atica, Universidade de Bras\'\i lia, 70910-900 Bras\'\i lia, DF, Brasil}

\email{alexei@unb.br}

\date{}

\begin{document}

\maketitle

\begin{abstract}
Let $F$ be a field of characteristic $\ne 2,3$ and let $A$ be a unital associative $F$-algebra. Define a left-normed commutator $[a_1, a_2, \dots , a_n]$ $(a_i \in A)$ recursively by $[a_1, a_2] = a_1 a_2 - a_2 a_1$, $[a_1, \dots , a_{n-1}, a_n] = [[a_1, \dots , a_{n-1}], a_n]$ $(n \ge 3)$. For $n \ge 2$, let $T^{(n)} (A)$ be the two-sided ideal in $A$ generated by all commutators $[a_1, a_2, \dots , a_n]$ ($a_i \in A )$. Define $T^{(1)} (A) = A$.

Let $k, \ell$ be integers such that $k > 0$, $0 \le \ell \le k$. Let $m_1, \dots , m_k$ be positive integers such that $\ell$ of them are odd and $k - \ell $ of them are even. Let $N_{k, \ell} =  \sum_{i=1}^k m_i -2k + \ell + 2 $. The aim of the present note is to show that, for any positive integers  $m_1, \dots , m_k$, in general,   $ T^{(m_1)} (A) \dots T^{(m_k)} (A) \nsubseteq T^{(1 + N_{k, \ell} )} (A)$. It is known that if $\ell < k$ (that is, if at least one of $m_i$ is even) then $T^{(m_1)} (A) \dots T^{(m_k)} (A) \subseteq T^{(N_{k, \ell} )} (A)$ for each $A$ so our result cannot be improved if $\ell <k$.

Let $N_k =  \sum_{i=1}^k m_i  -k+1$. Recently Dangovski has proved that if $m_1, \dots , m_k$ are any positive integers  then, in general, $T^{(m_1)} (A) \dots T^{(m_k)} (A) \nsubseteq T^{(1 + N_k)} (A) $. Since $ N_{k, \ell} = N_k - (k - \ell -1)$, Dangovski's result is stronger than ours if $\ell = k$ and is weaker than ours if $\ell \le k-2$; if $\ell = k-1$ then $N_k = N_{k, k-1}$ so both results coincide. It is known that if $\ell = k$ (that is, if all $m_i$ are odd) then, for each $A$,  $T^{(m_1)} (A) \dots T^{(m_k)} (A) \subseteq T^{(N_k)} (A)$
so in this case Dangovski's result cannot be improved.
\end{abstract}




\section{Introduction}

Let $R$ be an arbitrary unital associative and commutative ring and let $A$ be a unital associative algebra over $R$. Define a left-normed commutator $[a_1, a_2, \dots , a_n]$ ($a_i \in A$) recursively by $[a_1, a_2] = a_1 a_2 - a_2 a_1$, $[a_1, \dots , a_{n-1}, a_n] = [[a_1, \dots , a_{n-1}], a_n]$ $(n \ge 3)$. For $n \ge 2$, let $T^{(n)} (A)$ be the two-sided ideal in $A$ generated by all commutators $[a_1, a_2, \dots , a_n]$ ($a_i \in A )$. Define $T^{(1)} (A) = A$. Clearly, we have 
\[
A = T^{(1)} (A) \supseteq T^{(2)} (A) \supseteq T^{(3)} (A) \supseteq \dots \supseteq T^{(n)} (A) \supseteq \dots . 
\]

We are concerned with the following.

\medskip
\noindent
\textbf{Problem 1.}
\textit{
Let $k \ge 2$ and let $m_1, \dots ,m_k$ be positive integers. Find the maximal integer $N = N(R,m_1, \dots ,m_k)$ such that, for each $R$-algebra $A$,
}
\[
T^{(m_1)} (A) \dots T^{(m_k)} (A) \subseteq  T^{(N)} (A).
\]

Let $X = \{x_1, x_2, \dots \} $ be an infinite countable set and let $R \langle X \rangle$ be the free unital associative algebra over $R$ freely generated by $X$. Define $T^{(n)} = T^{(n)} ( R \langle X \rangle )$.  

\medskip
\noindent
\textbf{Problem 2.}
\textit{
Let $k \ge 2$ and let $m_1, \dots ,m_k$ be positive integers. Find the maximal integer $N = N(R,m_1, \dots ,m_k)$ such that 
}
\[
T^{(m_1)}  \dots T^{(m_k)}  \subseteq  T^{(N)} .
\]

It is easy to check that Problem 1 is equivalent to Problem 2, and the integer $N$ in both problems is the same.

Problem 2 and some other similar questions have been recently studied by Dangovski \cite{Dangovski15} (using different terminology). The work of Dangovski  was motivated by the results of Etingof, Kim and Ma \cite{EKM09} and Bapat and Jordan \cite{BJ10}, which in turn were motivated by the pioneering article by Feigin and Shoikhet \cite{FS07}.

The following assertion was proved by Latyshev \cite[Lemma 1]{Latyshev65} in 1965 (Latyshev's paper was published in Russian) and independently rediscovered by Gupta and Levin \cite[Theorem 3.2]{GL83} in 1983.
\begin{theorem}[see \cite{GL83,Latyshev65}]
\label{LGL}
Let $R$ be an arbitrary unital associative and commutative ring and let $A$ be an associative $R$-algebra. Let $m, n \in \mathbb Z,$ $m, n \ge 1$. Then 
\[
T^{(m)} (A) \  T^{(n )} (A) \subseteq T^{(m+n -2)} (A).
\]
\end{theorem}

Latyshev \cite{Latyshev65} has actually proved that $T^{(m)}  \  T^{(n )}  \subseteq T^{(m+n -2)} $ in $R \langle X \rangle$; this assertion is equivalent to Theorem \ref{LGL}.

Note that, for a unital associative ring $R$, we have $\frac{1}{6} \in R$ if and only if $2 (=1+1)$ and $3$ are invertible in $R$. The theorem below was proved by Sharma and Srivastava \cite[Theorem 2.8]{SS90} in 1990 and independently rediscovered (with different proofs) by Bapat and Jordan \cite[Corollary 1.4]{BJ10} in 2013 and by Grishin and Pchelintsev \cite[Theorem 1]{GrishinPchel15} in 2015. 

\begin{theorem}[see \cite{BJ10,GrishinPchel15,SS90}]
\label{BJ}
Let $R$ be an arbitrary unital associative and commutative ring such that $\frac{1}{6} \in R$ and let $A$ be an associative $R$-algebra. Let $m, n \in \mathbb Z,$ $m, n >1$ and at least one of the numbers $m$, $n$ is odd. Then
\begin{equation*}
\label{tmtn}
T^{(m)} (A) \  T^{(n)} (A) \subseteq T^{(m + n -1)} (A).
\end{equation*}
\end{theorem}
Note that  Grishin and Pchelintsev \cite{GrishinPchel15} have actually proved that $T^{(m)}  \  T^{(n)}  \subseteq T^{(m + n -1)} $; this result is equivalent to Theorem \ref{BJ}.

Let $N_k = \sum_{i=1}^k m_i - k + 1$. The proposition below follows immediately from Theorem \ref{BJ}.
\begin{proposition}
\label{odd}
Let $R$ be an arbitrary unital associative and commutative ring such that $\frac{1}{6} \in R$ and let $A$ be an associative $R$-algebra. Let $k>0$ be an integer and let $m_i > 0$ $(i=1, \dots , k)$ be \textbf{odd} integers. Then 
\[
T^{(m_1)} (A) \dots T^{(m_k)} (A) \subseteq  T^{(N_k )} (A).
\]
\end{proposition}

Let $N_{k, \ell} = \sum_{i=1}^k m_i -2k + \ell + 2 = N_k - (k-\ell -1)$. One can deduce from Theorems \ref{LGL} and \ref{BJ} the following proposition (see Dangovski \cite[Section 6]{Dangovski15}). 

\begin{proposition}[see \cite{Dangovski15}]
\label{lemmalowerbound}
Let $R$ be an arbitrary unital associative and commutative ring such that $\frac{1}{6} \in R$ and let $A$ be an associative $R$-algebra. Let $k, \ell$ be integers such that $0 \le \ell <k$. Let $m_i \ge 2$ $(i=1, \dots , k)$ be integers such that $\ell$ of them are odd and $(k - \ell )> 0$ of them are even. Then 
\begin{equation}
\label{lowerbound}
T^{(m_1)} (A) \dots T^{(m_k)} (A) \subseteq T^{(N_{k, \ell} )} (A) .
\end{equation}
\end{proposition}
We prove Proposition \ref{lemmalowerbound} in Section 2 in order to have the paper more self-contained.

Recently Dangovski \cite[Proposition 2.2]{Dangovski15} has proved a result that can be reformulated as follows.
\begin{theorem}[see \cite{Dangovski15}]
\label{D}
Let $F$ be a field and let $k$ be a positive integer. Let $m_1, \dots , m_k$ be positive integers and let $N_k$ be as above. Then there exists an associative $F$-algebra $A$ such that
\begin{equation}
\label{notin-odd}
T^{(m_1)} (A) \dots T^{(m_k)} (A) \nsubseteq T^{(1 + N_k )} (A) .
\end{equation}
\end{theorem}
One can deduce from Theorem \ref{D} the following.
\begin{corollary}
\label{D-corollary}
Let $R$ be an arbitrary unital associative and commutative ring and let $k, m_1, \dots , m_k, N_k$ be as in Theorem \ref{D}. Then there exists an associative $R$-algebra $A$ such that (\ref{notin-odd}) holds.
\end{corollary}
\begin{proof} Suppose that $R$ is not a field. Let $M$ be a maximal ideal of $R$ (by Zorn's lemma, such an ideal $M$ exists). Then $F = R/M$ is a field and the $F$-algebra $A$ of Theorem \ref{D} can be viewed in a natural way as an $R$-algebra (with $r \cdot a$ defined by $r \cdot a = (r + M) \cdot a$ for $r \in R, a \in A$). Since $A$ satisfies (\ref{notin-odd}), the result follows.
\end{proof} 

Let $N$ be the integer defined in Problems 1 and 2. If $\frac{1}{6} \in R$ and all the integers $m_1, \dots , m_k$ are \textit{odd} then $N=N_k$. Indeed, it follows from Proposition \ref{odd} and 
Corollary \ref{D-corollary}  that  in this case we always have 
\begin{equation*}
\label{iscontained}
T^{(m_1)} (A) \dots T^{(m_k)} (A) \subseteq T^{(N_k)} (A) 
\end{equation*}
and, in general, 
\begin{equation*}
\label{isnotcontained}
T^{(m_1)} (A) \dots T^{(m_k)} (A) \nsubseteq T^{(1 + N_k)} (A) .
\end{equation*}

Suppose that $\ell$ of the integers $m_1, \dots , m_k$ are odd $( \ell < k)$ and $(k-\ell ) >0 $ of them are even. Let $\frac{1}{6} \in R$.  Then, by Proposition  \ref{lemmalowerbound}, $N_{k, \ell} \le N$ and, by Corollary \ref{D-corollary}, $N \le N_k$. If $\ell = k-1$ (that is, $k-1$ of the integers $m_1, \dots , m_k$ are odd and one of them is even) then $N_{k, k-1}  = N_k $ so $N = N_k$. However, if $0 \le \ell <k-1$ then $N_{k, \ell} = N_k - (k - \ell -1) < N_k$ so one can only deduce from the results above that $N_{k, \ell} \le N \le N_k$.

Our main result is as follows.
\begin{theorem}
\label{maintheorem1}
Let $F$ be a field. Let $k, \ell $ be integers, $0 \le \ell \le k$. Let $m_1, \dots , m_k$ be positive integers such that $\ell$ of them are odd and $k - \ell $ of them are even and let $N_{k, \ell}$ be as above. Then there exists a unital associative $F$-algebra $A$ such that
\begin{equation}
\label{notin-even}
T^{(m_1)} (A) \dots T^{(m_k)} (A) \nsubseteq T^{(1 + N_{k, \ell})} (A).
\end{equation}
\end{theorem}
In a particular case when $k=2$ and $m_1$, $m_2$ are even Theorem \ref{maintheorem1} has been recently proved by Grishin and Pchelintsev \cite{GrishinPchel15} and independently by the authors of the present article \cite{DK17}. In another particular case when $m_1=m_2 = \dots = m_{k-1} =2$ and $m_k$ is even this theorem has been proved by Grishin, Tsybulya and Shokola \cite[Theorem 3]{GTS11}.

The proof of the following result is similar to that of Corollary \ref{D-corollary}.
\begin{corollary}
\label{maintheorem1-corollary}
Let $R$ be an arbitrary unital associative and commutative ring and let $k$, $\ell$, $m_1, \dots , m_k$, $N_{k, \ell}$ be as in Theorem \ref{maintheorem1}. Then there exists an associative $R$-algebra $A$ such that (\ref{notin-even}) holds.
\end{corollary}

It follows that if $\frac{1}{6} \in R$ and at least one of the integers $m_i$ is even then $N = N_{k, \ell}$ because, by Proposition \ref{lemmalowerbound}  and  Corollary \ref{maintheorem1-corollary},  in this case we always have
\[
T^{(m_1)} (A) \dots T^{(m_k)} (A) \subseteq T^{(N_{k, \ell} )} (A) 
\]
but, in general, 
\[
T^{(m_1)} (A) \dots T^{(m_k)} (A) \nsubseteq T^{(1 + N_{k, \ell})} (A).
\]

Thus, the solution of Problems 1 and 2 (for $R$ that contains $\frac{1}{6}$) is as follows. Let $R$ be a unital associative and commutative ring such that $\frac{1}{6} \in R$ and let $k, m_1, \dots , m_k$ be positive integers.  Then
\[
N = \left\{
\begin{array}{ll}
N_k = \sum_{i=1}^k m_i - k+1  & \mbox{ if all integers $m_i$ are odd (Dangovski \cite{Dangovski15})};
\\

\\
N_{k, \ell} = \sum_{i=1}^k m_i - 2k + \ell + 2 & 
\mbox{ 
\begin{minipage} {60mm} 
if $\ell <k$ of the integers $m_i$ are odd and $k-\ell$ of them are even.
\end{minipage}
}
\end{array}
\right .
\]

Recall that an associative algebra $A$ is Lie nilpotent of class at most $c$ if $[u_1, \dots , u_c , u_{c+1}]=0$ for all $u_i \in A$. Theorem \ref{maintheorem1} follows immediately from the following result.

\begin{theorem}
\label{maintheorem2}
Under the hypotheses of Theorem \ref{maintheorem1}, there exists a  unital associative $F$-algebra $A$ such that the following two conditions are satisfied:

i) $T^{(1 + N_{k, \ell})} (A) = 0$, that is, the algebra $A$ is Lie nilpotent of class at most $N_{k, \ell};$

ii) there are $v_{ij} \in A$ such that 
\[
[v_{11}, \dots , v_{1 m_1}] \dots [v_{k1}, \dots , v_{k m_k}]  \ne 0.
\]
\end{theorem}

To prove Theorem \ref{maintheorem2} we  use the same algebra $A$ that was used in \cite[Theorem 1.4]{DK17}.

\bigskip
\noindent
\textbf{Remarks.}
1. Both Theorem \ref{D} and Theorem \ref{maintheorem1} are valid for arbitrary $k$-tuples $m_1, m_2, \dots , m_k$ of positive integers. However, if $\ell = k$ (that is, if all $m_i$ are odd) then Theorem \ref{D} gives a stronger result than Theorem \ref{maintheorem1}  because $N_{k,k} = N_k +1 > N_k$ and therefore $T^{(1 + N_{k, k})} (A) \subset T^{(1 + N_k)} (A)$. If $\ell = k-1$ (that is, if one of the integers $m_1, m_2, \dots , m_k$ is even and $k-1$ of them are odd) then $N_{k, k-1} = N_k$ so the results of Theorem \ref{D} and Theorem \ref{maintheorem1} coincide; and if $\ell < k-1$ (that is, if two or more of the integers $m_1, m_2, \dots , m_k$ are even) then $N_{k, \ell} = N_k - (k - \ell - 1) < N_k$ so Theorem \ref{D} gives a weaker result than Theorem \ref{maintheorem1}. 

2. The proofs of Theorem \ref{BJ} given in \cite{BJ10}, \cite{GrishinPchel15} and \cite{SS90} are valid for algebras over  an associative and commutative unital ring $R$ such that $\frac{1}{6} \in R$. However, the proof given in \cite{BJ10} can be slightly modified to become also valid over any $R$ such that $\frac{1}{3} \in R$ (see \cite[Remark 3.9]{AE15} for explanation). Moreover, for some specific $m$ and $n$ Theorem \ref{BJ} holds over an arbitrary ring $R$: for instance, $T^{(3)} (A) T^{(3)} (A) \subset T^{(5)} (A)$ for any algebra $A$ over any associative and commutative unital ring $R$ (see \cite[Lemma 2.1]{CostaKras13}). However, in general Theorem \ref{BJ} fails over $\mathbb Z$ and over a field of characteristic $3$: it was shown in \cite{DK15,Kr13} that in this case $T^{(3)} T^{(2)} \nsubseteq T^{(4)}$ and moreover, $T^{(3)} \bigl( T^{(2)}\bigr) ^{\ell} \nsubseteq T^{(4)}$ for all $\ell \ge 1$.

3. In 1978  Volichenko proved Theorem \ref{BJ} for $m=3$ and arbitrary $n$ in the preprint \cite{Volichenko78} written in Russian. In 1985 Levin and Sehgal \cite{LS85} independently rediscovered Volichenko's result. More recently Etingof, Kim and Ma \cite{EKM09} and Gordienko \cite{Gordienko07} have independently proved this theorem for small $m$ and $n$; these authors were unaware of the results of \cite{LS85,Volichenko78}.


\section{Proofs of Proposition \ref{lemmalowerbound} and Theorem \ref{maintheorem2} }

\begin{proof}[Proof of Proposition \ref{lemmalowerbound}] Induction on $k$. If $k=1$ then $\ell = 0$ so $N_{1, 0} = m_1$ and (\ref{lowerbound}) holds.

Suppose that $k>1$ and for all products of less than $k$ terms $T^{(m_i)} (A)$ the proposition has already been proved. We split the proof in 3 cases.

Case 1. Suppose that $m_k$ is odd. Then for some $i$ such that $1 \le i < k$ the number $m_i$ is even so we can apply the induction hypothesis to the product $T^{(m_1)} (A) \dots T^{(m_{k-1})} (A)$. By this hypothesis, 
\[
T^{(m_1)} (A) \dots T^{(m_{k-1})} (A) \subseteq T^{(N')} (A)
\]  
where $N' = \sum_{i=1}^{k-1} m_i - 2(k-1) +(\ell -1) +2 = \sum_{i=1}^{k-1} m_i -2k + \ell +3$. By Theorem \ref{BJ}, 
\[
T^{(N')} (A) \ T^{(m_k)}(A) \subseteq T^{(N' +m_k -1)} (A) = T^{(N_{k, \ell})} (A)
\]
since $N' + m_k -1 = \sum_{i=1}^k m_i -2k + \ell +2 = N_{k, \ell}$. Thus, in this case (\ref{lowerbound}) holds, as required.

Case 2. Suppose that $m_k$ is even and, for some $i$ such that $1 \le i <k$, $m_i$ is also even. Then we can apply the induction hypothesis to the product $T^{(m_1)} (A) \dots T^{(m_{k-1})} (A)$ so
\[
T^{(m_1)} (A) \dots T^{(m_{k-1})} (A) \subseteq T^{(N'')} (A)
\]
where $N'' =  \sum_{i=1}^{k-1} m_i - 2(k-1) +\ell +2 = \sum_{i=1}^{k-1} m_i -2k + \ell +4$. By Theorem \ref{LGL}, 
\[
T^{(N'')} (A) \ T^{(m_k)} (A) \subseteq T^{(N'' + m_k -2)} (A)  = T^{(N_{k, \ell})} (A)
\]
since $N'' + m_k -2 =  \sum_{i=1}^{k} m_i -2k + \ell +2 = N_{k, \ell}$. Hence, in this case (\ref{lowerbound}) holds, as required.

Case 3. Suppose that $m_k$ is even and all $m_i$ for $1 \le i < k$ are odd. Applying Theorem \ref{BJ} $k-1$ times, we get
\[
T^{(m_1)} (A) \dots T^{(m_{k-1})} (A) \ T^{(m_k)} (A) \subseteq T^{(m_1 + \dots + m_k - k +1)} (A) = T^{(N_{k, k-1} )} (A)
\]
since $\sum_{i=1}^k m_i -k + 1 = \sum_{i=1}^k m_i -2k + (k-1) +2 = N_{k, k-1}$. Thus, in this case (\ref{lowerbound}) also holds.

The proof of Proposition \ref{lemmalowerbound} is completed.
\end{proof}

The proof of  Theorem \ref{maintheorem2}  below is a modification of the proof of \cite[Theorem 1.4]{DK17}. First we need some auxiliary results.

Let $G$ and $H$ be unital associative algebras over a field $F$ such that $[g_1, g_2, g_3] = 0$, $[h_1, h_2, h_3] =0$ for all $g_i \in G$, $h_j \in H$. Note that each commutator $[g_1, g_2]$ $(g_i \in G)$ is central in $G$, that is, $[g_1, g_2] g = g [g_1, g_2]$ for each $g \in G$. Similarly, each commutator $[h_1,h_2]$ $(h_j \in H)$ is central in $H$. The following lemma has been proved in \cite[Lemma 2.1]{DK17} by induction on $n$.

\begin{lemma}[see \cite{DK17}]
\label{cl}
Let
\[
c_{\ell} = [ g_1 \otimes h_1, g_2 \otimes h_2, \dots , g_{\ell} \otimes h_{\ell}] 
\]
where $\ell \ge 2, g_i \in G, h_j \in H$. Then
\begin{align*}
c_2 = & \ [g_1,g_2] \otimes h_1 h_2 + g_2 g_1 \otimes [h_1, h_2],
\\
c_{2n} = & \ [g_1,g_2] [g_3,g_4] \dots [g_{2n-1},g_{2n}] \otimes [h_1 h_2, h_3] [h_4,h_5] \dots [h_{2n-2}, h_{2n-1}] h_{2n} 
\\
+ & \ [g_2 g_1, g_3] [g_4,g_5] \dots [g_{2n-2}, g_{2n-1}] g_{2n} \otimes [h_1, h_2] [h_3,h_4] \dots [h_{2n-1}, h_{2n}] 
\\
& (n >1),
\\
c_{2n+1} = & \ [g_1, g_2] [g_3,g_4]  \dots [g_{2n-1}, g_{2n}] g_{2n+1} \otimes [h_1 h_2, h_3][h_4,h_5] \dots [h_{2n}, h_{2n+1}]
\\
+ & \ [g_2 g_1, g_3] [g_4, g_5] \dots [g_{2n}, g_{2n+1}] \otimes [h_1, h_2] [h_3,h_4] \dots [h_{2n-1}, h_{2n}] h_{2n+1} 
\\
& (n \ge 1).
\end{align*}
\end{lemma}

\begin{corollary}[see \cite{DK17}]
\label{nilp}
Suppose that 
\begin{equation}
\label{prod}
[f_1, f_2] \dots [f_{2n-1}, f_{2n}] = 0  \qquad \mbox{for all} \ \ f_j \in H.
\end{equation}
Then for all $u_i \in G \otimes H$ we have
\[
[u_1, u_2, \dots , u_{2n+1}] = 0.
\]
\end{corollary}

\begin{proof}
 It follows from (\ref{prod}) and  Lemma \ref{cl} that $[ g_1 \otimes h_1, g_2 \otimes h_2, \dots , g_{2n+1} \otimes h_{2n+1}] = 0$ for all $g_i \in G$, $h_j \in H$. Since each $u_i \in G \otimes H$ is a sum of products of the form $g \otimes h$ ($g \in G$, $h \in H$), we have $[u_1, u_2, \dots , u_{2n+1}] = 0$ for all $u_i \in G \otimes H$, as required.
\end{proof}

The following assertion follows immediately from Lemma \ref{cl}.
\begin{corollary}
\label{nilp2}
Let $v_1 = g_1 \otimes 1$, $v_i = g_i \otimes h_i$ $(i = 2, \dots , 2m'-1)$, $v_{2m'}=g_{2m'} \otimes 1$ and let $w_1 = g_1' \otimes 1$, $w_j = g_j' \otimes h_j'$ $(j = 2, \dots , 2n'+1)$ where $g_i,g_i' \in G$, $h_j, h_j' \in H$. Then
\begin{align*}
[v_1, \dots , v_{2m'}]  =  & \ [g_1,g_2] \dots [g_{2m'-1}, g_{2m'}] \otimes   \ [h_2, h_3] \dots [h_{2m'-2},h_{2m'-1}],
\\
[w_1, \dots , w_{2n' + 1}] = & \ [g_1', g_2'] \dots [g_{2n'-1}', g_{2n'}']  g_{2n'+1}  \otimes \ [h_2',h_3'] \dots [h_{2n'}',h_{2n'+1}'] .
\end{align*}
\end{corollary}

\begin{proof}[Proof of Theorem \ref{maintheorem2}] Two cases are to be considered: the case when $char \  F \ne 2$ and the case when $char \ F =2$.

Case 1. Suppose that $F$ is a field of characteristic $\ne 2$. Let $E$ be the unital infinite-dimensional Grassmann (or exterior) algebra over $F$. Then $E$ is generated by the elements $e_i$ $(i = 1, 2, \dots )$ such that $e_i e_j = - e_j e_i$, $e_i^2 = 0$ for all $i, j$ and the set
\[
\mathcal B = \{ e_{i_1} e_{i_2} \dots e_{i_k} \mid k \ge 0, \, i_1 < i_2 < \dots < i_k \}
\]
forms a basis of $E$ over $F$. It is well known and easy to check that $[g_1,g_2,g_3]=0$ for all $g_i \in E$. 

Recall that the $r$-generated unital Grassmann algebra $E_r$ is the unital subalgebra of $E$ generated by $e_1, e_2, \dots , e_r$. Note that $[h_1,h_2,h_3]=0$ for all $h_j \in E_r$. 

Take $A=E \otimes E_r$ where $r = \sum_{i=1}^k m_i -2k +\ell = N_{k, \ell} -2. $ It is easy to check that $r$ is an even integer.  We can apply Lemma \ref{cl} and Corollaries \ref{nilp} and \ref{nilp2} for $G=E$, $H=E_r$.

Note that $[f_1,f_2] \dots [f_{r+1},f_{r+2}]=0$ for all $f_i \in E_r$. Indeed, for all $f, f' \in E_r$ the commutator $[f,f']$ belongs to the linear span of the set $\{ e_{i_1}  \dots e_{i_{2 \ell}} \mid \ell \ge 1, 1 \le i_s \le r \}$.  Hence, $[f_1,f_2] \dots [f_{r+1},f_{r+2}]$ belongs to the linear span of the set $\{ e_{i_1}  \dots e_{i_{2 \ell}} \mid \ell \ge (r+2)/2, 1 \le i_s \le r \}$. Since $2 \ell \ge r+2  > r$, each product $e_{i_1}  \dots e_{i_{2 \ell}}$ above contains equal terms $e_{i_{s}} = e_{i_{s'}}$ $(s < s')$ and, therefore, is equal to $0$. Thus, $[f_1,f_2] \dots [f_{r+1},f_{r+2}]=0$, as claimed. 

Since $N_{k, \ell} = r +2$, we have $[f_1,f_2] \dots [f_{(N_{k, \ell} -1)},f_{N_{k, \ell}}]=0$ for all $f_i \in E_r$. Hence, by Corollary \ref{nilp}, we have $[u_1, \dots , u_{(1 + N_{k, \ell} )}]=0$ for all $u_i \in A = E \otimes E_r$, that is, $T^{(1 + N_{k, \ell})} (A) = 0$, as required. 

Now it suffices to find elements $v_{ij} \in A$ such that 
\begin{equation}
\label{product}
[v_{11}, \dots , v_{1 m_1}] \dots [v_{k1}, \dots , v_{k m_k}]  \ne 0.
\end{equation}
Let 
\[
\mathcal P = \{ (i,j) \mid 1 \le i \le k; \ 1 \le j \le m_i \}.
\]
Note that $v_{ij}$ appears in (\ref{product}) if and only if $(i,j) \in \mathcal P$. Let $ \mathcal N = \sum_{i=1}^k m_i $ and let $\mu : \mathcal P \rightarrow \{ 1,2, \dots , \mathcal N \}$ be a bijection. Define 
\[
e_{ij} = e_{\mu (i,j)} \qquad \bigl( (i,j) \in \mathcal P \bigr) .
\]
Note that 
\begin{equation}
\label{prod1}
\prod_{(i,j) \in \mathcal P} e_{ij} = (-1)^{\delta} e_1 e_2 \dots e_{\mathcal N}
\end{equation}
for some $\delta \in \{ 0, 1 \}$.

Let $\mathcal P' \subset \mathcal P$, 
\[
\mathcal P' = \{ (i',j') \mid 1 \le i' \le k; \ 2 \le j' \le m_i - 1 \mbox{ if $m_i$ is even}; \ 2 \le j' \le m_i \mbox{ if $m_i$ is odd} \} .
\]
Let $\mu' : \mathcal P' \rightarrow \{ 1,2, \dots , \sum_{i=1}^k m_i - 2k + \ell \} = \{ 1,2, \dots , r \}$ be a bijection. Define
\[
e'_{i' j'} = e_{\mu' (i', j')} \qquad \bigl( (i',j' ) \in \mathcal P' \bigr) .
\]
Note that 
\begin{equation}
\label{prod2}
\prod_{(i',j') \in \mathcal P'} e_{i'j'} = (-1)^{\delta'} e_1 e_2 \dots e_r
\end{equation}
for some $\delta' \in \{ 0, 1 \}$.

Define
\begin{align*}
v_{i1}  & =   \ e_{i1} \otimes 1;
\\
v_{ij}   & =    \ e_{ij} \otimes e'_{ij} \qquad (1 \le i \le k; \ 2 \le j \le m_i -1) ;
\\
v_{i m_i} &  = \  \left\{
\begin{array}{ll}
e_{i {m_i}} \otimes 1 & \mbox{if $m_i$ is even};
\\
e_{i {m_i}} \otimes  e'_{i {m_i}} & \mbox{if $m_i$ is odd}.
\end{array}
\right .
\end{align*}
If $m_i$ is even then, by Corollary \ref{nilp2}, 
\begin{align*}
& [v_{i 1}, v_{i 2}, \dots , v_{i m_{i}}] 
\\
= \ & [e_{i1}, e_{i2}] [e_{i3}, e_{i4}] \dots [e_{i (m_i-1)}, e_{i m_i}] \otimes [e'_{i2}, e'_{i3}] [e'_{i4}, e'_{i5}] \dots [e'_{i (m_i - 2)}, e'_{i (m_i - 1)}] .
\end{align*}
Note that $e_{st} e_{s't'} = - e_{s't'} e_{st}$  for all $s, s', t, t'$ so $[e_{st}, e_{s't'}] = 2e_{st} e_{s't'}$. It follows that if $m_i$ is even then
\[
[v_{i 1}, v_{i 2}, \dots , v_{i m_{i}}] = 2^{m_i -1} e_{i1} e_{i2} \dots e_{i m_i} \otimes e'_{i2} e'_{i3} \dots e'_{i (m_i-1)} .
\]
If $m_i$ is odd then, by Corollary \ref{nilp2}, 
\begin{align*}
& [v_{i 1}, v_{i 2}, \dots , v_{i m_{i}}] 
\\
= \ & [e_{i1}, e_{i2}] [e_{i3}, e_{i4}] \dots [e_{i (m_i-2)}, e_{i (m_i -1)}] e_{i m_i} \otimes [e'_{i2}, e'_{i3}] [e'_{i4}, e'_{i5}] \dots [e'_{i (m_i - 1)}, e'_{i m_i }]
\\
= \ & 2^{m_i -1} e_{i1} e_{i2} \dots  e_{i (m_i -1)} e_{i m_i} \otimes e'_{i2} e'_{i3} \dots e'_{i (m_i - 1)} e'_{i m_i } .
\end{align*}
It follows that
\[
[v_{11}, \dots , v_{1 m_1}] \dots [v_{k1}, \dots , v_{k m_k}] 
=  2^{N_k -1} \prod_{i=1}^k \prod_{j=1}^{m_i} e_{ij} \otimes \prod_{i'=1}^k \prod_{j'=2}^{m_{i'}'} e'_{i' j'} 
\]
where 
\[
m'_{i'} = \left\{
\begin{array}{ll}
m_{i'} - 1 & \mbox{ if $m_{i'}$ is even};
\\
m_{i'} & \mbox{ if $m_{i'}$ is odd},
\end{array}
\right .
\]
that is, 
\[
[v_{11}, \dots , v_{1 m_1}] \dots [v_{k1}, \dots , v_{k m_k}] = 2^{N_k -1} \prod_{(i,j) \in \mathcal P} e_{ij} \otimes \prod_{(i',j') \in \mathcal P'}  e'_{i' j'} . 
\]
By (\ref{prod1}) and (\ref{prod2}), we have 
\[
[v_{11}, \dots , v_{1 m_1}] \dots [v_{k1}, \dots , v_{k m_k}] = (-1)^{\delta + \delta'} 2^{N_k -1}  \ e_1 e_2 \dots e_{\mathcal N} \otimes e_1 e_2 \dots e_r \ne 0, 
\]
as required.


\medskip
Case 2. Suppose that $F$ is a field of characteristic $2$. Let $\mathcal G$ be the group given by the presentation
\[
\mathcal G = \langle y_1, y_2, \dots \mid y_i^2 =1, \ \big( (y_i, y_j), y_k \big) = 1 \ (i,j,k = 1,2, \dots ) \rangle
\]
where $(a,b)=a^{-1}b^{-1} a b$. Then it is easy to check that $\mathcal G$ is a nilpotent group of class $2$ so $(a,b)c = c (a,b)$ for all $a,b,c \in \mathcal G$ and, therefore, $(a, bc) = (a,c) c^{-1} (a,b) c = (a,b) (a,c)$ (see \cite{DK17} for more details). It is clear that the quotient group $\mathcal G / \mathcal G'$ is an elementary abelian $2$-group so $b^2 \in \mathcal G' \subseteq Z (\mathcal G )$ for all $b \in \mathcal G$. It follows that $(a,b^2) =1$ so $(a,b)^2 = (a,b^2) =1$, that is, $(a,b) = (a, b)^{-1}$. Since  $(b,a) = (a,b)^{-1}$, we have $(a,b) = (b,a)$ for all $a,b \in \mathcal G$.  

Let $(<)$ be an arbitrary linear order on the set $\{ (i,j) \mid i,j \in \mathbb Z, \ 0< i <j \}$. The following lemma is well known and easy to check.

\begin{lemma}
\label{groupG}
Let $a \in \mathcal G$. Then $a$ can be written in a unique way  in the form
\begin{gather}
\label{formG}
a = y_{i_1} \dots y_{i_q} (y_{j_1}, y_{j_2}) \dots (y_{j_{2q'-1}}, y_{j_{2q'}}) 
\\
 \mbox{where }\ q,q' \ge 0;  \ i_1< \dots < i_q, \ j_{2s-1}< j_{2s} \mbox{ for all }  s,  \nonumber
 \\ (j_{2s-1}, j_{2s}) <  (j_{2s'-1}, j_{2s'}) \mbox{ if } \ s < s' . \nonumber
\end{gather}
\end{lemma}

Let $F \mathcal G$ be the group algebra of $\mathcal G$ over $F$. Let $d_{ij} = (y_i,y_j) + 1 \in F \mathcal G$. Note that $d_{ij} = d_{ji}$ and $d_{ii} = 0$ for all $i,j$.

Let $I$ be the two-sided ideal of $F \mathcal G$ generated by the set
\[
S = \{ d_{i_1 i_2}  d_{i_3 i_4} + d_{i_1 i_3} d_{i_2 i_4} \mid i_1, i_2, i_3, i_4 = 1,2 \dots \} .
\]

The following two lemmas are well known (see, for instance, \cite[Lemma 2.1]{GK95}, \cite[Example 3.8]{GL83});  their proofs can also be found in \cite{DK17}. 

\begin{lemma}
\label{nilp22}
For all $u_1, u_2, u_3 \in F \mathcal G$, we have $[u_1, u_2, u_3] \in I$. 
\end{lemma}

\begin{lemma}
\label{notin}
For all $\ell  >0$, we have 
\[
\bigl( ({y}_{1}, {y}_{2}) + 1 \bigr)  \bigl( ({y}_{3}, {y}_{4}) + 1 \bigr) \dots \bigl(  ({y}_{{2 \ell -1}}, {y}_{{2\ell}}) + 1 \bigr)  \notin I .
\]
\end{lemma}

Since the ideal $I$ is invariant under all permutations of the set $\{ y_1, y_2, \dots \}$ of generators of the group $\mathcal G$, we have the following.

\begin{corollary}
\label{corollary-notin}
Let $\ell >0$. Then $\bigl( ({y}_{i_1}, {y}_{i_2}) + 1 \bigr) \dots \bigl(  ({y}_{i_{2 \ell -1}}, {y}_{i_{2\ell}}) + 1 \bigr)  \notin I $ if all integers $i_1, i_2, \dots , i_{2 \ell}$ are distinct.
\end{corollary} 


Now we are in a position to complete the proof of Theorem \ref{maintheorem2}. Recall that $r = \sum_{i=1}^k m_i -2k +\ell  = N_{k, \ell} - 2$ is an even integer. Let $\mathcal G_r$ be the subgroup of $\mathcal G$ generated by $y_1, \dots , y_r$; let $I_r = I \cap F \mathcal G_r$. Take $G = F \mathcal G / I$, $H = F \mathcal G_r /I_r$. Take $A = G \otimes H$. By Lemma \ref{nilp22}, we can apply Lemma \ref{cl} and Corollaries \ref{nilp} and \ref{nilp2} to $A$.

We claim that  $[f_1,f_2] \dots [f_{r+1},f_{r+2}] \in I_r$ for all $f_i \in F \mathcal G_r$. Indeed, we may assume without loss of generality that $f_i \in \mathcal G_r$ for all $i$. Since 
\begin{align*}
& [f_{2s-1}, f_{2s}] = f_{2s-1} f_{2s} + f_{2s} f_{2s-1} 
\\
= \ & f_{2s-1} f_{2s} \bigl( (f_{2s}, f_{2s-1} ) + 1 \bigr)  =  f_{2s-1} f_{2s} \bigl( (f_{2s-1}, f_{2s} ) + 1 \bigr)
\end{align*}
(recall that $F$ is a field of characteristic $2$), we have 
\[
[f_1,f_2] \dots [f_{r+1},f_{r+2}] = f_1 f_2 \dots f_{r+2} \bigl( (f_1,f_2)+1 \bigr) \dots \bigl( (f_{r+1}, f_{r+2}) + 1 \bigr) .
\]
It is clear that, for each $s$, $(f_{2s-1}, f_{2s}) = \prod_t c_{i_{st} j_{st}} $ for some commutators $c_{i_{st} j_{st}} = (y_{i_{st}}, y_{j_{st}})$. Let $d_{i_{st} j_{st}} = c_{i_{st} j_{st}} + 1$; then $c_{i_{st} j_{st}} = d_{i_{st} j_{st}} + 1$.  We have
\begin{align*}
& (f_{2s-1}, f_{2s}) + 1 = \prod_t c_{i_{st} j_{st}}  +1  = \Big( \prod_t ( d_{i_{st} j_{st}}   +1) \Big) +1 
\\
=  \ & \prod_t d_{i_{st} j_{st}}   + \dots +  \sum_{t<t'} d_{i_{st} j_{st}}  d_{i_{st'} j_{st'}}  + \sum_t d_{i_{st} j_{st}} .
\end{align*}
It follows that the product $\bigl( (f_1,f_2)+1 \bigr) \dots \bigl( (f_{r+1}, f_{r+2}) + 1 \bigr)$ can be written as a sum of products of the form 
\begin{equation}
\label{product2}
d_{q_1 q_2} \dots d_{q_{2 \ell -1} q_{2 \ell}} =   \bigl( (y_{q_1}, y_{q_2}) + 1 \bigr) \dots \bigl( (y_{q_{2 \ell -1}}, y_{q_{2 \ell}}) + 1 \bigr)
\end{equation}
where $2\ell \ge r+2 >r$. Hence, in the product (\ref{product2}) we have $q_{t}= q_{t'}$ for some $t < t'$. 

Note that $d_{j_1 j_3} d_{j_2 j_3}  \in I$ for all $j_1,j_2,j_3$ because $d_{j_1 j_3} d_{j_2 j_3} = d_{j_1 j_3} d_{j_2 j_3} + d_{j_1 j_2} d_{j_3 j_3}  \in S$.  Since $d_{i j} = d_{j i}$ for all $i,j$, we have $ d_{i_1 i_2} d_{i_3 i_4}  \in I$ if any two of the indices $i_1, i_2, i_3, i_4$ coincide. It follows that each product (\ref{product2}) belongs to $I_r = I \cap F \mathcal G_r$ and so does the product $\bigl( (f_1,f_2)+1 \bigr) \dots \bigl( (f_{r+1}, f_{r+2}) + 1 \bigr)$. Hence,  $[f_1,f_2] \dots [f_{r+1},f_{r+2}] \in I_r$, as claimed. Since $N_{k, \ell} = r +2$, we have $[f_1,f_2] \dots [f_{(N_{k, \ell} -1)},f_{N_{k, \ell}}] \in I_r$ for al $f_i \in F \mathcal G_r$.

For any $u \in F \mathcal G$, let $\bar{u} = u + I \in G = F \mathcal G /I$. Since one can view the algebra $H = F \mathcal G_r /I_r$ as a subalgebra of $G= F \mathcal G /I$, we also write $\bar{u} = u + I_r \in H = F \mathcal G_r /I_r$ for $u \in F \mathcal G_r$. 

By the observation above, 
$[\bar{f}_1,\bar{f}_2] \dots [\bar{f}_{(N_{k, \ell} - 1)},\bar{f}_{N_{k, \ell}}]=0$ for all $\bar{f}_i \in H$. Hence, by Corollary \ref{nilp}, we have $[u_1, \dots , u_{(1 + N_{k, \ell} )}]=0$ for all $u_i \in A = G \otimes H$, that is, $T^{(1 + N_{k, \ell})} (A) = 0$, as required. 

Let $\mathcal P$, $\mathcal P'$, $\mu$ and $\mu'$ be as in Case 1. Recall that $\mathcal N = \sum_{i=1}^k m_i$. Define 
\[
y_{ij} = y_{\mu (i, j)} \quad \bigl( (i, j) \in \mathcal P \bigr), 
\qquad \qquad
y'_{i' j'} = y_{\mu' (i', j')} \quad  \bigl( (i', j') \in \mathcal P' \bigr).
\]  
Define
\begin{align*}
v_{i1}  & =   \ \bar{y}_{i1} \otimes 1;
\\
v_{ij}   & =    \ \bar{y}_{ij} \otimes \bar{y}'_{ij} \qquad (1 \le i \le k; \ 2 \le j \le m_i -1) ;
\\
v_{i m_i} &  = \  \left\{
\begin{array}{ll}
\bar{y}_{i {m_i}} \otimes 1 & \mbox{if $m_i$ is even};
\\
\bar{y}_{i {m_i}} \otimes  \bar{y}'_{i {m_i}} & \mbox{if $m_i$ is odd}.
\end{array}
\right .
\end{align*}
If $m_i$ is even then, by Corollary \ref{nilp2}, 
\begin{align*}
& [v_{i 1}, v_{i 2}, \dots , v_{i m_{i}}]  
\\
= \ & [\bar{y}_{i1}, \bar{y}_{i2}] [\bar{y}_{i3}, \bar{y}_{i4}] \dots [\bar{y}_{i (m_i-1)}, \bar{y}_{i m_i}] \otimes [\bar{y}'_{i2}, \bar{y}'_{i3}] [\bar{y}'_{i4}, \bar{y}'_{i5}] \dots [\bar{y}'_{i (m_i - 2)}, \bar{y}'_{i (m_i - 1)}]
\\
= \  & \bar{y}_{i1} \bar{y}_{i2} \bar{y}_{i3} \dots \bar{y}_{i m_i} \bigl( ( \bar{y}_{i1}, \bar{y}_{i2}) + 1 \bigr) \bigl( ( \bar{y}_{i3}, \bar{y}_{i4}) + 1 \bigr)\dots \bigl( (\bar{y}_{i (m_i-1)}, \bar{y}_{i m_i}) +1 \bigr) 
\\
\otimes \ & \bar{y}'_{i2} \bar{y}'_{i3} \dots \bar{y}'_{i (m_i - 1)} \bigl( (\bar{y}'_{i2}, \bar{y}'_{i3}) + 1 \bigr) \bigl( (\bar{y}'_{i4}, \bar{y}'_{i5}) + 1\bigr) \dots \bigl( ( \bar{y}'_{i (m_i - 2)}, \bar{y}'_{i (m_i - 1)}) + 1 \bigr)
\end{align*}
If $m_i$ is odd then, by the same corollary, 
\begin{align*}
& [v_{i 1}, v_{i 2}, \dots , v_{i m_{i}}]  
\\
= \ & [\bar{y}_{i1}, \bar{y}_{i2}] [\bar{y}_{i3}, \bar{y}_{i4}] \dots [\bar{y}_{i (m_i-2)}, \bar{y}_{i (m_i -1)}] \bar{y}_{i m_i} \otimes [\bar{y}'_{i2}, \bar{y}'_{i3}] [\bar{y}'_{i4}, \bar{y}'_{i5}] \dots [\bar{y}'_{i (m_i - 1)}, \bar{y}'_{i m_i }]
\\
= \  & \bar{y}_{i1} \bar{y}_{i2} \bar{y}_{i3} \dots  \bar{y}_{i (m_i - 1)} \bar{y}_{i m_i} \bigl( ( \bar{y}_{i1}, \bar{y}_{i2}) + 1 \bigr) \bigl( ( \bar{y}_{i3}, \bar{y}_{i4}) + 1 \bigr)\dots \bigl( (\bar{y}_{i (m_i-2)}, \bar{y}_{i (m_i -1)}) +1 \bigr) 
\\
\otimes \ & \bar{y}'_{i2} \bar{y}'_{i3} \dots \bar{y}'_{i (m_i - 1)} \bar{y}'_{i m_i} \bigl( (\bar{y}'_{i2}, \bar{y}'_{i3}) + 1 \bigr) \bigl( (\bar{y}'_{i4}, \bar{y}'_{i5}) + 1\bigr) \dots \bigl( ( \bar{y}'_{i (m_i - 1)}, \bar{y}'_{i m_i }) + 1 \bigr)
\end{align*}
It follows that
\[
[v_{11}, \dots , v_{1 m_1}] \dots [v_{k1}, \dots , v_{k m_k}] = \bar{y} \ Q \otimes   \bar{y}' \ Q'
\]
where 
\[
\bar{y} =  \prod_{i=1}^k \prod_{j=1}^{m_i} \bar{y}_{ij} ,
\qquad
\bar{y}' = \prod_{i'=1}^k \prod_{j'=2}^{m_{i'}'} \bar{y}'_{i' j'},
\]
\[
m'_{i'} = \left\{
\begin{array}{ll}
m_{i'} - 1 & \mbox{ if $m_{i'}$ is even};
\\
m_{i'} & \mbox{ if $m_{i'}$ is odd},
\end{array}
\right .
\]
\[
Q = \prod_{i=1}^k \prod_{j=1}^{\bigl[ \frac{m_i}{2} \bigr]} \bigl( (\bar{y}_{i (2j-1)} , \bar{y}_{i (2j)}) + 1 \bigr) ,
\qquad
Q' = \prod_{i'=1}^k \prod_{j'=1}^{\bigl[ \frac{m'_{i'} - 1}{2} \bigr]} \bigl( (\bar{y}'_{i' (2j')} , \bar{y}'_{i' (2j'+1)}) + 1 \bigr)  .
\]
Since $\mu$ is injective, all elements $y_{i (2j - 1)}, y_{i (2j)}$ $(i = 1,2, \dots , k; \ j = 1,2, \dots , \bigl[ \frac{m_i}{2} \bigr] )$ that appear in $Q$ are distinct elements of the set $\{ y_1, y_2, \dots \}$. Hence, by Corollary \ref{corollary-notin}, we have $Q \ne 0$ in $G = F \mathcal G /I$. Similarly, $Q' \ne 0$ in $H = F \mathcal G_r /I_r$. Since $\bar{y}$ and $\bar{y}'$ are invertible elements of $G$ and $H$, respectively, we have $\bar{y} \ Q \otimes \bar{y}' \ Q' \ne 0$, that is, 
\[
[v_{11}, \dots , v_{1 m_1}] \dots [v_{k1}, \dots , v_{k m_k}]  \ne 0, 
\]
as required.

This completes the proof of Theorem \ref{maintheorem2}.
\end{proof}

\noindent
\textbf{Remark.}
Recall that in the proof of Theorem \ref{maintheorem2}  we use the same algebra $A$ that was used in the proof of \cite[Theorem 1.4]{DK17}. Note that in both proofs one can choose the algebra $A$ different from one used in our proofs. For example, let $F$ be any field and let $r = m+n-4= 2(m'+n'-2)$. Let $A = F \langle X \rangle /T^{(3)} \otimes F \langle X_r \rangle /T^{(3)}_r$ where $X_r = \{ x_1, \dots , x_r \}$ and $T^{(3)}_r = T^{(3)} \bigl( F \langle X_r \rangle \bigr) =  T^{(3)} \cap F \langle X_r \rangle.$ Then $A$ satisfies the conditions i) and ii) of Theorem \ref{maintheorem2}; one can check this using a description of a basis of $F \langle X \rangle /T^{(3)}$ over $F$. Such a description can be deduced, for instance, from \cite[Proposition 3.2]{BEJKL12} or found (if $char \ F \ne 2$) in \cite[Proposition 9]{BKKS10}.

Our choice of the algebra $A$ in the proof of Theorem \ref{maintheorem2} was made with a purpose to have the paper self-contained. 

\section*{Acknowledgments}

This work was partially supported by CNPq grant 310331/2015-3 and by RFBR grant 15-01-05823. We thank Victor Petrogradsky for pointing out the references \cite{LS85,SS90}.



\begin{thebibliography}{99}

\bibitem{AE15}
N. Abughazalah, P. Etingof, On properties of the lower central series of associative algebras, \textit{J. Algebra Appl.} \textbf{15} (2016) 1650187 (24 pages).   arXiv:1508.00943 [math.RA].

\bibitem{BJ10}
A. Bapat, D. Jordan, Lower central series of free algebras in symmetric tensor categories, \textit{J. Algebra} \textbf{373} (2013) 299--311. arXiv:1001.1375 [math.RA].

\bibitem{BEJKL12}
S. Bhupatiraju, P. Etingof, D. Jordan, W. Kuszmaul and J. Li, Lower central series of a free associative algebra over the integers and finite fields, \textit{J. Algebra} \textbf{372} (2012) 251--274. arXiv:1203.1893 [math.RA].

\bibitem{BKKS10}
A. Brand\~ao Jr., P. Koshlukov, A. Krasilnikov, E.A. Silva, The central polynomials for the Grassmann algebra, \textit{Israel J. Math.} \textbf{179} (2010) 127--144.

\bibitem{CostaKras13}
E.A. da Costa, A. Krasilnikov, Relations in universal Lie nilpotent associative algebras of class $4$, to appear in \textit{Comm. Algebra}. DOI:
10.1080/00927872.2017.1347661.  arXiv:1306.4294 [math.RA].

\bibitem{Dangovski15}
R.R. Dangovski, On the maximal containments of lower central series ideals, arXiv:1509.08030 [math.RA].

\bibitem{DK15}
G. Deryabina, A. Krasilnikov, The torsion subgroup of the additive group of a Lie nilpotent associative ring of class $3$, \textit{J. Algebra} \textbf{428} (2015) 230--255. arXiv:1308.4172 [math.RA].

\bibitem{DK17}
G. Deryabina, A. Krasilnikov, Products of commutators in a Lie nilpotent associative algebra, \textit{J. Algebra} \textbf{469 } (2017) 84--95. arXiv:1509.08890 [math.RA].

\bibitem{EKM09}
P. Etingof, J. Kim, X. Ma, On universal Lie nilpotent associative algebras, \textit{J. Algebra} \textbf{321} (2009) 697--703. arXiv:0805.1909 [math.RA].

\bibitem{FS07}
B. Feigin, B. Shoikhet, On $[A, A]/[A, A, A]$ and on a $W_n$-action on the consecutive commutators of free associative algebras, \textit{Math. Res. Lett.}  \textbf{14} (2007) 781--795. arXiv:math/0610410.

\bibitem{Gordienko07}
A.S. Gordienko, Codimensions of commutators of length $4$, \textit{Russian Math. Surveys} \textbf{62} (2007) 187--188.

\bibitem{GrishinPchel15}
A.V. Grishin, S.V. Pchelintsev, On the centers of relatively free algebras with an identity of Lie nilpotency, \textit{Sb. Math.} \textbf{206} (2015) 1610--1627. 

\bibitem{GTS11}
A.V. Grishin, L.M. Tsybulya, A.A. Shokola, On $T$-spaces and relations in relatively free, Lie nilpotent, associative algebras, \textit{J. Math. Sci. (N.Y.)} \textbf{177} (2011) 868--877.

\bibitem{GK95}
C.K. Gupta, A.N. Krasilnikov, Some non-finitely based varieties of groups and group representations, \textit{Internat. J. Algebra Comput.} \textbf{5} (1995) 343--365. 

\bibitem{GL83}
N. Gupta, F. Levin, On the Lie ideals of a ring, \textit{J. Algebra} \textbf{81} (1983) 225--231.

\bibitem{Kr13}
A. Krasilnikov, The additive group of a Lie nilpotent associative ring, \textit{J. Algebra} \textbf{392} (2013) 10--22. arXiv:1204.2674 [math.RA].

\bibitem{Latyshev65}
V.N. Latyshev, On the finiteness of the number of generators of a T-ideal with an element $[x_1,x_2,x_3,x_4]$, \textit{Sibirsk. Mat. Zh.} \textbf{6} (1965) 1432--1434 (in Russian). 

\bibitem{LS85}
F. Levin, S. Sehgal, On Lie nilpotent group rings, \textit{J. Pure Appl. Algebra} \textbf{37} (1985) 33--39.

\bibitem{SS90}
R.K. Sharma, J.B. Srivastava, Lie ideals in group rings, \textit{J. Pure Appl. Algebra} \textbf{63} (1990) 67--80. 

\bibitem{Volichenko78}
I.B. Volichenko, The $T$-ideal generated by the element $[x_1, x_2, x_3, x_4]$, Preprint no. 22, Inst. Math. Acad. Sci. Beloruss. SSR (1978) (in Russian).


\end{thebibliography}
\end{document}